\documentclass[11pt]{amsart}
\usepackage{amsfonts}
\usepackage{amsmath}
\usepackage{enumerate}
\usepackage{tikz}
\usepackage{bm}

\usepackage{geometry}
\newlength{\ten}
\newcommand{\de}{\partial}
\newcommand{\vp}{\varphi}
\newcommand{\db}{\overline{\partial}}
\newcommand{\im}{\sqrt{-1}}
\newcommand{\ddbar}{\im\de\db}

\newcommand{\ov}[1]{\overline{#1}}
\newcommand{\tr}[2]{\mathrm{tr}_{#1}{#2}}
\newcommand{\Ric}{\mathrm{Ric}}
\newcommand{\Laplace}{\Delta}

\newcommand{\newsort}[1]{}

\numberwithin{equation}{section}

\newtheorem{thm}{Theorem}[section]

\newtheorem{prop}[thm]{Proposition}
\newtheorem{lem}[thm]{Lemma}
\newtheorem{cor}[thm]{Corollary}

\theoremstyle{definition}

\theoremstyle{definition}
\newtheorem{remark}[thm]{Remark}

\begin{document}

\title{The Chern-Ricci flow on primary Hopf surfaces}
\author{Gregory Edwards}

\begin{abstract}
The Hopf surfaces provide a family of minimal non-K\"ahler surfaces of class VII on which little is known about the Chern-Ricci flow. We use a construction of Gauduchon-Ornea for locally conformally K\"ahler metrics on primary Hopf surfaces of class 1 to study solutions of the Chern-Ricci flow. These solutions reach a volume collapsing singularity in finite time, and we show that the metric tensor satisfies a uniform upper bound, supporting the conjecture that the Gromov-Hausdorff limit is isometric to a round $S^1$. Uniform $C^{1+\beta}$ estimates are also established for the potential. Previous results had only been known for the simplest examples of Hopf surfaces.
\end{abstract}

\maketitle

\section{Introduction}

The Chern-Ricci flow is a parabolic flow of Hermitian metrics first studied by Gill \cite{Gi11} and later introduced in greater generality by Tosatti-Weinkove \cite{ToWe15}. We say $g(t)$ is a solution to the Chern-Ricci flow starting from a Hermitian metric $g_0$ if
\begin{equation} \label{eq:20}
	\begin{cases}
		\frac \de {\de t} g = - \Ric^{\mathcal C h}(g) \\
		g(0) = g_0
	\end{cases}
\end{equation}
 where $\Ric^{\mathcal Ch}$ is the \textit{Chern-Ricci} tensor of $g$ defined by
\[
	\Ric^{\mathcal Ch}_{i \ov j} = - \de_i \de_{\ov j} \log \det g.
\]

If the associated $(1,1)$-form, $\omega_0 = \im (g_0)_{i \ov j} dz^i \wedge d \ov z^j$ is closed, then $g_0$ is a K\"ahler metric and the Chern-Ricci tensor is equal to the usual Ricci tensor. Thus the Chern-Ricci flow yields the same solution as the well known K\"ahler-Ricci flow \cite{Cao,ChWa12,FIK,PSSW09,PS06,SoTi07,SoTi12,SoTi09,SoWe11b,SoWe11,SoWe13,SoYu12,Sz10,Ti08,TiZh,To10,ToWeYa14,Ts}. Other flows of Hermitian metrics have also been proposed and studied \cite{StTi10,StTi11,StTi13,LiYa12,Ust16,Ust17}.

One direction of interest introduced in \cite{ToWe13} is to classify the behavior of the Chern-Ricci flow of Gauduchon metrics on complex surfaces. On complex surfaces, a \textit{Gauduchon} metric is a Hermitian metric whose associated $(1,1)$-form satisfies $\de \db \omega_0 = 0$. A well known result of Gauduchon states that every Hermitian metric lies in the conformal class of a Gauduchon metric \cite{Ga77}. Furthermore any Hermitian metric in the $\de\db$-class,
\[
	\mathcal H_{\omega_0} = \{ \omega_\psi | \omega_\psi = \omega_0 + \ddbar \psi > 0 \textit{ for } \psi \in C^{\infty}(M) \},
\]
is also Gauduchon. On surfaces, the Gauduchon condition is preserved by the Chern-Ricci flow \cite{ToWe15} and the Chern-Ricci flow of Gauduchon metrics on complex surfaces has been studied in several contexts \cite{GiSm13,ToWe13,ToWe15,ToWeYa15}.

For surfaces which are not minimal (i.e. those which have exceptional divisors) and with Kodaira dimension not equal to $-\infty$, the flow reaches a finite time non-collapsing singularity at which time it contracts finitely many disjoint exceptional curves in the Gromov-Hausdorff topology, up to a condition on the $\de \db$-class of the limiting form \cite{ToWe13,ToWe15}, generalizing results for the K\"ahler-Ricci flow \cite{SoWe11b,SoWe11,SoWe11a}. 

By the Enriques-Kodaira classification of complex surfaces \cite{BHPV04}, all minimal non-K\"ahler surfaces can be classified into the following families:
\begin{center}
\begin{minipage}{.6\textwidth}
\begin{enumerate}[(i)]
	\item Kodaira surfaces,
	\item Minimal non-K\"{a}hler properly elliptic surfaces,
	\item Inoue surfaces,
	\item Hopf surfaces
	\item Minimal surfaces of class VII with $b_2(M)>0$,
\end{enumerate}
\end{minipage}
\end{center}
where Kodaira surfaces are minimal surfaces with $b_1(M)$ odd and Kodaira dimension 0; Inoue surfaces are those with universal cover $\mathbb C \times H$ where $H$ is the upper half plane; Hopf surfaces are those with universal cover $\mathbb C^2 \setminus \{0\}$; and surfaces of class VII are surfaces with $b_1(M)=1$ and Kodaira dimension $-\infty$. By \cite{Bo82,Ko66,LiYaZh94,Te94}, a class VII surface with $b_2(M)=0$ must be either a Hopf or Inoue surface.

Solutions of the Chern-Ricci flow have been studied in several of the cases above: On manifolds with vanishing first Bott-Chern class -- in any dimension -- the flow converges smoothly to a Chern-Ricci flat Hermitian metric \cite{Gi11} using the uniform $C^0$-estimate of \cite{ToWe10}; on minimal non-K\"ahler elliptic surfaces the normalized Chern-Ricci flow converges in the Gromov-Hausdorff topology to an orbifold K\"ahler-Einstein metric on a Riemann surface \cite{ToWeYa15}; and on Inoue surfaces, after a conformal change to the initial metric, the Chern-Ricci flow converges in the Gromov-Hausdorff topology to a round $S^1$ up to scaling \cite{FaToWeZh16}. The surfaces of type (v) are not yet classified except for the case $b_2(M)=1$ \cite{Te05} and one long-term goal of study for the Chern-Ricci flow is to provide new topological or geometric information about Class VII surfaces in general.

On Hopf surfaces, the flow always reaches a finite time singularity at which time the volume goes to zero \cite{ToWe15}. Beyond this, little is currently known about the Chern-Ricci flow on Hopf surfaces in any generality. The round metric on $S^3 \times S^1$ admits a compatible complex structure as a Hopf surface, and the Chern-Ricci flow of this metric has an explicit maximal solution \cite{ToWe15}. The solution becomes extinct at time $T = \frac 1 2$, and $(M,g(t))$ converges in the Gromov-Hausdorff topology to a round $S^1$ up to a scaling factor \cite{ToWe13}. Moreover, if the initial metric is in the same $\de \db$-class as the round metric, then the solution satisfies an upper bound and the potential converges in $C^{1+\beta}$ for every $\beta \in (0,1)$ \cite{ToWe15}.

The primary Hopf surfaces of class 1, as defined in \cite{HaLa83}, form a large class of Hopf surfaces. These are defined as the quotients $M = M_{\alpha, \beta} = (\mathbb C^2 \setminus \{ 0 \})/\sim$ by the action $(z_1,z_2) \mapsto (\alpha z_1, \beta z_2)$ for $\alpha, \beta \in \mathbb C$, with $1 < |\alpha| \leq |\beta|$. All primary Hopf surfaces\footnote{The primary Hopf surfaces consist of both those of class 1, and those of class 0 which are defined as quotients of $\mathbb C^2 \setminus \{0\}$ of the form $(z_1,z_2) \mapsto (\beta^m z_1 + \lambda z_2^m, \beta z_2)$ for some positive integer $m$ and $\beta,\lambda \in \mathbb C$ with $1 < |\beta|$ and $\lambda \neq 0$.} are diffeomorphic to $S^3 \times S^1$, and all Hopf surfaces are finitely covered by a primary Hopf surface \cite{Ko66,Ko66a}. In particular, since the second Betti number vanishes, it is clear these surfaces do not admit any K\"ahler metric.

While $M$ is never K\"ahler, one can construct Hermitian metrics on $M$ which are \textit{locally conformally} K\"ahler. The existence of such metrics was first proved by LeBrun (see \cite{GaOr98}), for $|\alpha|,|\beta|$ distinct but sufficiently close, and explicit examples were constructed by Gauduchon-Ornea \cite{GaOr98}. These particular metrics are of interest because they provide examples of Hermitian metrics on these surfaces. It is a difficult problem in general to give explicit Hermitian metrics on Hopf surfaces, particularly ones for which $\alpha \neq \beta$.

These LCK metrics are constructed as follows: We define a function, $\Phi = \Phi(|z_1|,|z_2|)$, on $\mathbb C^2 \setminus \{0\}$ which satisfies the relation
\begin{equation}\label{eq:1}
	|z_1|^2 \Phi^{-2k_1} + |z_2|^2 \Phi^{-2k_2} = 1,
\end{equation}
where
\[
	k_1 = \frac {\log |\alpha|}{\log |\alpha| + \log |\beta|} \leq k_2 = \frac {\log |\beta|}{\log |\alpha| + \log |\beta|}.
\]
Indeed for any real constants $a,b$, not both zero,
\[
	s \mapsto a^2 s^{-2k_1} + b^2 s^{-2k_2}
\]
is continuous and strictly decreasing from $\infty$ to $0$ for $s>0$, and hence there is a unique value for which $a^2 s^{-2k_1} + b^2 s^{-2k_2} = 1$.

While $\Phi$ is a well defined function on $\mathbb C^2 \setminus \{ 0 \}$, it does not define a function on $M$. However, the $(1,1)$-form
\begin{equation}\label{eq:5}
	\hat \omega = \frac {\ddbar \Phi}{\Phi}
\end{equation}
is well defined and positive definite (see Remark ~\ref{eq:9}), and hence it defines a Hermitian metric on $M$. These Hermitian metrics are never K\"ahler, but they are locally conformally K\"ahler (LCK), and satisfy
\begin{equation}\label{eq:4}
	d \hat \omega =  \hat \omega \wedge \theta
\end{equation}
for a closed, real 1-form given by
\begin{equation}\label{eq:23}
	\theta = \frac {d \Phi}{\Phi}.
\end{equation}
The existence of a closed 1-form satisfying ~\eqref{eq:4} is equivalent to the Hermitian metric being LCK \cite{Ga84,Lee43,Va80}.

As a special case of ~\eqref{eq:5}, when $|\alpha|=|\beta|$, we recover $\Phi  = r^2 = |z_1|^2 + |z_2|^2$, and $(M,\hat \omega)$ is isometric to the round metric on $S^3 \times S^1$. We call this \textit{the round metric on the standard Hopf surface}, and $\omega(t) = \hat \omega - t \Ric^{\mathcal Ch}(\hat \omega)$ provides an explicit maximal solution to the Chern-Ricci flow ~\eqref{eq:20} on $M \times [0,\frac 1 2)$ converging in the Gromov-Hausdorff topology to a round $S^1$ \cite{ToWe13,ToWe15}.

The LCK condition ~\eqref{eq:4} is not preserved under the flow, even for the round metric on the standard Hopf surface. However, we show that these initial metrics are Gauduchon (See Corollary (\ref{eq:10corb})), and therefore any metric in their $\de \db$-class is also Gauduchon. Since the Chern-Ricci flow preserves the Gauduchon condition it is of interest to study solutions starting from the $\de \db$-class of the LCK metrics on non-standard primary Hopf surfaces defined above.

Our main theorem is the following:
\begin{thm}\label{thm:main}
	Let $\hat \omega$ be the LCK metric constructed above, and $\omega_0 = \hat \omega + \ddbar \psi$ for some smooth plurisubharmonic function $\psi$ with $g_0$ its the associated Hermitian metric. Then a maximal solution to the Chern-Ricci flow ~\eqref{eq:20} exists on the time interval $[0,\frac 12)$ and there is a uniform constant $C>0$, independent of $t$, such that
	\[
		\omega(t) \leq C \hat \omega.
	\]
	on $M \times [0, \frac 12)$.
\end{thm}

Unlike the case for the round metric on the standard Hopf surface, we do not obtain an explicit solution to the Chern-Ricci flow from any initial starting metric. Indeed, it seems such solutions are very difficult to find explicitly, and consequently there are difficulties at present in controlling the Gromov-Hausdorff limit of the solutions.


From the bound of the trace we also obtain the following result on the convergence of the potential.
\begin{cor}
	Set $\omega(t) = \hat \omega - t \Ric(\hat \omega) + \ddbar \psi$ with $\psi$ normalized to satisfy equation ~\eqref{eq:28} below. Then as $ t \to T^-$, $\psi(t)$ converges subsequentially to a function $\psi(T)$ in $C^{1+\beta}$ for every $\beta \in (0,1)$.
\end{cor}

This follows from the fact that by the estimates in Lemma \ref{lem:1} below, after passing to subsequence, $\psi(t)$ converges pointwise to a function $\psi(T)$ as $t \to T^-$, and by Theorem \ref{thm:main}, $|\Laplace_{\hat \omega} \psi|$ is uniformly bounded, and so $\| \psi(t) \|_{C^{1+\beta}}$ is uniformly bounded for any $\beta \in (0,1)$. It follows that, after passing to subsequence, $\psi(t) \to \psi(T)$ in $C^{1+\beta}$ as $t \to T^-$ for any $\beta \in (0,1)$.

\section*{Outline}

The outline of the rest of the paper is as follows. In Section 2, we establish some geometric properties of the LCK metrics defined above and show that they satisfy the Gauduchon condition. In Section 3, we formulate the Chern-Ricci flow as a parabolic complex Monge-Amp\`ere equation, and recall the uniform estimate on the potential, and an upper bound on its time derivative. In Section 4, we bound the trace of the evolving metric with respect to the LCK metric and complete the proof of Theorem ~\ref{thm:main}.

\section*{Acknowledgments}

Research for this paper began at the American Institute of Mathematics workshop: Nonlinear PDEs in real and complex geometry in San Jose, CA August 2018. The author thanks AIM for their hospitality. The author also extends their thanks to Casey Kelleher, Valentino Tosatti, Yury Ustinovskiy, and Ben Weinkove for helpful discussions at the AIM workshop. The author was supported by the NSF grant RTG: Geometry and Topology at the University of Notre Dame.

\section{Geometry of the locally conformally K\"ahler metrics}

In order to compute various geometric quantities related to the LCK metrics on non-standard Hopf surfaces we first compute the form of the metric in coordinates. We define the following (1,1)-form related to ~\eqref{eq:23},
\begin{equation}\label{eq:24}
	\Theta =\im \theta^{(1,0)} \wedge \theta^{(0,1)} = \im \frac {\de \Phi \wedge \db \Phi}{\Phi^2}.
\end{equation}
Clearly, $\Theta$ is closed, non-negative, real, and of rank one. To obtain the components of the metric, we proceed as follows using $\Phi_i$ as a shorthand for $\de_i \Phi$, $\Phi_{i \ov j} = \de_i \de_{\ov j} \Phi$, \textit{etc}.: 

First, differentiating ~\eqref{eq:1},
\begin{equation}\label{eq:26}
	\Phi_i = \ov z_i \Phi^{1-2k_i} Z^{-1}
\end{equation}
where
\begin{equation}\label{eq:27}
	Z = 2 \big( k_1 |z_1|^2 \Phi^{-2k_1} + k_2 |z_2|^2 \Phi^{-2k_2} \big).
\end{equation}
Note that $Z$ descends to a well defined function on $M$ which satisfies
\begin{equation}\label{eq:30}
	2 k_1 \leq Z \leq 2 k_2.
\end{equation}
We compute
\[
	Z_i = 2 \ov z_i k_i \Phi^{-2k_i} - 4 \ov z_i \Phi^{-2k_i} Z^{-1} \sum_a |z_a|^2 k_a^2 \Phi^{ -2k_a},
\]
so that
\[
	\Phi_{i \ov j} = \delta_{ij} \frac{\Phi^{1-2k_i}}{Z}  + (1 - 2k_i - 2k_j)  \ov z_i z_j\frac{ \Phi^{1-2k_i-2k_j}}{Z^2} + 4 \ov z_i z_j \frac {\Phi^{1-2k_i-2k_j}}{Z^3} \sum_a |z_a|^2 k_a^2 \Phi^{ -2k_a},
\]
and we have
\begin{equation}\label{eq:11}
	\Phi_{i \ov j} = \delta_{i j} \frac {\Phi^{1 - 2k_i}} Z + (1 - 2 k_i - 2k_j + \frac 4 Z \sum_a k_a |z_a|^2 \Phi^{-2k_a} ) \frac {\Phi_i \Phi_{\ov j} }{\Phi}.
\end{equation}

Next, we compute the determinant of $\hat g$.

\begin{prop}\label{eq:22}
The determinant of $\hat g$ is given by the following identity,
\[
	\det(\hat g) = \frac 1 {\Phi^2 Z^3}
\]
\end{prop}

\begin{proof}
	The proof is contained in Gauduchon-Ornea \cite{GaOr98}. We provide it here, adapted to our slightly different conventions, for convenience. We first compute the individual components of the complex Hessian of $\Phi$, for instance:
	\begin{align*}
		\Phi_{1 \ov 1} = \frac {\Phi^{1 - 2k_1}} Z +& \big(1 - 4k_1 + \frac 4 Z (k_1^2 |z_1|^2 \Phi^{-2k_1} \\
		& + k_2^2 |z_2|^2 \Phi^{-2k_2}) \big) \frac {|z_1|^2 \Phi^{1 - 4k_1}}{Z^2} \\
		 = \frac {\Phi^{1-2k_1}} {Z^3} (& Z^2 +  (1 - 4k_1) Z |z_1|^2 \Phi^{-2k_1}  \\
		 & +  4 k_1^2 |z_1|^4 \Phi^{-4k_1} + 4 k_2^2 |z_1|^2 |z_2|^2 \Phi^{-2}) \\
		 = \frac {2\Phi^{1-2k_1}}{Z^3} (& k_1|z_1|^4 \Phi^{-4k_1} + 2k_2^2 |z_2|^4 \Phi^{-4k_2} \\
		 & + k_2 (1+2k_2) |z_1|^2 |z_2|^2 \Phi^{-2}),
	\end{align*}
	and similarly we obtain
	\begin{align*}
	\Phi_{2 \ov 2} = \frac {2 \Phi^{1-2k_2}}{Z^3} 
		& (k_2 |z_2|^4 \Phi^{-4k_2} + 2k_1^2 |z_1|^4 \Phi^{-4k_1} \\
		& + k_1 (1+2k_1) |z_1|^2 |z_2|^2 \Phi^{-2}),\\
	\Phi_{1 \ov 2} = \frac {2 \ov z_1 z_2 \Phi^{-1}}{Z^3}
		& (k_1 - k_2)( k_1 |z_1|^2 \Phi^{-2k_1} - k_2 |z_2|^2 \Phi^{-2k_2} ).
\end{align*}
Then we find the determinant of the matrix,
\[
	A = \left[ 
	{\begin{array}{cc}
	\Phi_{1 \ov 1} & \Phi_{1 \ov 2}  \\
	\Phi_{2 \ov 1} & \Phi_{2 \ov 2} \\
	\end{array} } \right],
\]
to be
\begin{align*}
	\det A = \frac 8 {Z^6} & \big( k_1^3 |z_1|^8 \Phi^{-8 k_1} + k_2^3 |z_2|^8 \Phi^{-8k_2} \\
	& + 3 k_1 k_2 |z_1|^4 |z_2|^4 \Phi^4 \\
	& + k_1^2 (1 + 2 k_2) |z_1|^6 |z_2|^2 \Phi^{-4k_1 - 2} \\
	& + k_2(1 + 2 k_1) |z_1|^2 |z_2|^6 \Phi^{-4k_2 - 2}\big) \\
	& = \frac 1 {Z^3},
\end{align*}
and therefore we have
\[
	\det(\hat g) = \frac 1 {\Phi^2 Z^3}
\]
which was claimed.
\end{proof}

\begin{remark}\label{eq:9}
	From the calculations above, it follows that $\hat g_{i \ov j}$ has strictly positive determinant and, by inspection, has strictly positive trace. Since $\mathrm{dim}_{\mathbb C}(M) = 2$, it follows that $\hat g_{i \ov j}$ defines a positive definite Hermitian metric.
\end{remark}


Next, we have the following geometric identity.
\begin{prop}\label{eq:10}
The following equality holds for the trace of $\Theta$:
\[
	\tr {\hat \omega} {\Theta} = 1.
\]
\end{prop}

As an immediate and crucial consequence is the following Corollary, obtained from the calculation of the trace and non-negativity of $\Theta$.

\begin{cor}\label{eq:10cor}
We have the inequality of $(1,1)$-forms:
\[
	\Theta \leq \hat \omega.
\]
\end{cor}

\begin{proof}[Proof of Proposition \ref{eq:10}]
From ~\eqref{eq:24}
\[
	\Theta_{i \ov j} = \frac {\Phi_i \Phi_{\ov j}}{\Phi^2}.
\]
We use the identity 
\begin{align*}
	\tr{\hat \omega}{\Theta}
	& = 2 \frac { \hat \omega \wedge \Theta}{\hat \omega^2} \\
	& = \Phi^2 Z^3 \big( \Phi^{-3} \Phi_{2 \ov 2} \Phi_1 \Phi_{\ov 1} - \Phi^{-3} \Phi_{1 \ov 2} \Phi_2 \Phi_{\ov 1} - \Phi^{-3} \Phi_{2 \ov 1} \Phi_1 \Phi_{\ov 2} + \Phi^{-3} \Phi_{1 \ov 1} \Phi_2 \Phi_{\ov 2} \big),
\end{align*}
and then using the calculations above
\begin{align*}
	 = \frac {Z^3} \Phi \bigg( &  \Phi^{1-2k_1} \frac {2|z_1|^2} {Z^5} \big( k_2 |z_2|^4 \Phi^{-4k_2} + 2 k_1^2 |z_1|^4 \Phi^{-4k_1} + k_1(1 + 2k_1) |z_1|^2 |z_2|^2 \Phi^{-2} \big) \\  + 
& \frac {2 |z_2|^2}{Z^5} \Phi^{ 1-2k_2} \big( k_1 |z_1|^4 \Phi^{-4k_1} + 2 k_2^2 |z_2|^4 \Phi^{ -4k_2} + k_2(1+2k_2)|z_1|^2|z_2|^2 \Phi^{-2} \big) \\ 
 - & 4 \frac {|z_1|^2 |z_2|^2}{Z^5} \Phi^{-1} (k_1-k_2) \big(k_1 |z_1|^2 \Phi^{-2k_1} - k_2 |z_2|^2 \Phi^{-2 k_2} \big) \bigg) \\
	=  \frac 4 {Z^2} \big( & k_1^2 |z_1|^6 \Phi^{-6k_1} + k_2^2 |z_2|^6 \Phi^{-6k_2} \\
	 + & k_2 (k_1 + 1) |z_1|^2|z_2|^4 \Phi^{-2-2k_2} + k_1(1+k_2)|u|^4|v|^2 \Phi^{-2 - 2k_1} \big)\\
	= 1 \hskip0.2in &
\end{align*}
where we have used ~\eqref{eq:1} in the last line.
\end{proof}

Proposition \ref{eq:10} also allows us to obtain the following result for the LCK metrics.
\begin{cor}\label{eq:10corb}
	The metrics $\hat \omega$ satisfy the Gauduchon condition.
\end{cor}
\begin{proof}
	Indeed
	\begin{align*}
		\de \db \hat \omega 
		& = \de \db \big( \im \frac {\de \db \Phi} \Phi \big) \\
		& = \im \big( - \frac {\de \db \Phi \wedge \de \db \Phi}{\Phi^2} + 2 \frac {\de \Phi \wedge \db \Phi \wedge \de \db \Phi}{\Phi^3} ) \\
		& = \frac 1 \im \big( - \hat \omega^2 + 2 \Theta \wedge \hat \omega \big) \\
		& = \frac 1 \im \big( \tr {\hat \omega} \Theta - 1) \hat \omega^2 = 0
	\end{align*}
	which proves the claim.
\end{proof}

%
%

Let us now define another metric which will be useful for our purposes:
\[
	\chi_{i \ov j} = \Phi^{-2k_i} \delta_{ij}.
\]
One can check that $\chi$ transforms in the correct way to define a Hermitian metric. On the standard Hopf surface this is equal to the round metric, but otherwise it is distinct from $\hat \omega$.

The benefit of introducing the new metric is that 
\[
	\det(\chi) = \frac 1 {\Phi^2},
\]
and so its Chern-Ricci form is given by:
\begin{align*}
	\Ric(\chi) & = 2 \ddbar \log \Phi \\
	& = 2 \frac {\ddbar \Phi} \Phi - 2 \im \frac {\de \Phi \wedge \db \Phi}{\Phi^2} \\
	&= 2 \hat \omega - 2 \Theta \geq 0
\end{align*}
using Corollary ~\ref{eq:10cor} to obtain the inequality.

\section{The Chern-Ricci flow}\label{sec:1}

Let $\omega(t)$ be the solution to the Chern-Ricci flow ~\eqref{eq:20} starting from $\omega(0) = \hat \omega + \ddbar \psi$ for a smooth plurisubharmonic function $\psi$. Then, we can write the solution as
\[
	\omega(t) = \hat \omega - t(2 \hat \omega - 2 \Theta + 3 \ddbar \log Z) + \ddbar \psi(t)
\]
where $\psi(t)$ solves the parabolic complex Monge-Amp\`ere equation
\begin{equation}\label{eq:28}
	\begin{cases}
		\dot \psi  = \log \frac {(\hat \omega - t(2 \hat \omega - 2 \Theta + 3 \ddbar \log Z) + \ddbar \vp)^2}{\hat \omega^2} \\
		\psi(0) = \psi.
	\end{cases}
\end{equation}
But since $\log Z$ is a globally defined smooth function, we can write
\[
	\omega(t) = \hat \omega - 2t (\hat \omega - \Theta) + \ddbar \vp(t)
\]
by setting
\begin{equation}\label{eq:21}
	\vp(t) = \psi(t) - 3 t \log Z.
\end{equation}
Then since $\hat \omega^2 = \frac 1 {Z^3} \chi^2$,  $\vp$ satisfies the equation
\begin{equation}
	\begin{cases}
		\dot \vp = \log \frac{(\hat \omega - 2t(\hat \omega - \Theta) + \ddbar \vp)^2}{\hat \omega^2} - 3 \log Z = \log  \frac{(\hat \omega - 2t(\hat \omega - \Theta) + \ddbar \vp)^2}{\chi^2} \\
		\vp(0) = \psi.
	\end{cases}
\end{equation}
We define the family of reference metrics,
\[
	\omega_t = (1 - 2t) \hat \omega + 2 t \Theta,
\]
so that
\[
	\omega(t) = \omega_t + \ddbar \vp,
\]
and note that
\[
	\frac \de {\de t} \omega_t = -2\hat \omega + 2 \Theta = -\Ric^{\mathcal C h}(\chi).
\]

From Tosatti-Weinkove \cite{ToWe15} we have that the Chern-Ricci flow exists on a maximal time interval $[0,T)$ where $T$ depends only on the $\de \db$-class of $\omega_t$, and for Gauduchon metrics on complex surfaces, $T$ is given explicitly by
\[
	T = \sup \{ t | \int_M \omega_t^2 > 0 \text{, and } \int_D \omega_t > 0 \text{ for all irreducible divisors } D \text{ with } D^2 < 0 \}.
\]
It follows that $T=\frac 1 2$, since $M$ has no divisors with $D^2<0$ and
\begin{align*}
	\omega_t^2 & = (1-2t)^2 \hat \omega^2 + 4t (1-2t) \hat \omega \wedge \Theta \\
		& = (1-2t) \hat \omega^2 
\end{align*}
since $2 \hat \omega \wedge \Theta = (\tr {\hat \omega} \Theta) \hat \omega^2$.

We have the following estimates on the potential for the solutions.
\begin{lem}\label{lem:1}
	There exists a uniform constant $C>0$ such that for $t \in [0,T)$,
	\begin{enumerate}[(i)]
		\item $ | \psi(t) | + \dot \psi(t) \leq C $
		\item $ | \vp(t) | +  \dot \vp(t)  \leq C$
	\end{enumerate}
\end{lem}

\begin{proof}
	The proof of part $(i)$ is standard and is contained in Tosatti-Weinkove \cite{ToWe15} (in the K\"ahler setting the proof is due to Tian-Zhang \cite{TiZh}), we include it here for convenience. We use $\Laplace = \Laplace_\omega$ for the Laplacian with respect to $g(t)$.
	Applying the maximum principle to $(\psi - At)$ for a constant $A>0$, we have that at a point of maximum with $t>0$
	\[
		0 \leq \frac \de {\de t} (\psi - A t) \leq \log \frac {\omega_t^n}{\hat \omega^n} - A < 0 
	\]
	if $A$ is chosen sufficiently large. Hence the maximum occurs at $t=0$, and therefore we have the upper bound on $\psi$. The lower bound follows a similar argument.
	
	To obtain the upper bound for $\dot \psi$, we apply the maximum principle to 
	\[
		Q = t \dot \psi - \psi - 2 t
	\]
	so that
	\[
		(\frac \de {\de t} - \Laplace) Q =  - \tr \omega { \hat \omega} < 0.
	\]
	By the maximum principle, 
	\[
		\sup_M Q(\cdot,0) \geq \sup_M Q(\cdot, t)
	\]
	and it follows that $\dot \psi$ is uniformly bounded from above.
	
	Part $(ii)$ follows from part $(i)$ and ~\eqref{eq:21} since $\log Z$ is bounded.
\end{proof}

\section{Bound of the metric along the Chern-Ricci flow}

Since $\hat \omega$ is controlled by $\chi$ is suffices to bound $\tr \chi {\omega(t)}$. Let $\hat g_{i \ov j}$ be the Hermitian metric associated to $\hat \omega$, and $g_{i \ov j}$ the metric associated to $\omega(t)$. We often use the Hermitian metrics and their associated $(1,1)$-forms interchangeably.

Let us fix the notation that $\Ric$ will denote the Chern-Ricci tensor of $\chi$, and so
\begin{equation}\label{eq:31}
	2 \hat g = \Ric + 2 \Theta.
\end{equation}

First, we note that 
\begin{align*}
	\big( \frac \de {\de t} - \Laplace_\omega \big) \tr \chi \omega = & - g^{i \ov j} \de_i \de_{\ov j} \chi^{k \ov l} g_{k \ov l} + \chi^{k \ov l} g^{i \ov j} (\de_k \de_{\ov l} \hat g_{i \ov j} - \de_i \de_{\ov j} \hat g_{i \ov j} ) \\
	& - 2 \mathrm{Re}(g^{i \ov j} \de_i \chi^{k \ov l} \de_{\ov j} g_{k \ov l}) - \chi^{k \ov l} g^{p \ov j} g^{i \ov q} \de_k g_{p \ov q} \de_{\ov l } g_{i \ov j}.
\end{align*}

Indeed, since
\begin{align*}
	\frac \de {\de t} \tr \chi {\omega(t)} & = \chi^{k \ov l} \de_k \de_{\ov l} \log \det g \\
	& = - \chi^{k \ov l} g^{p \ov j} g^{i \ov q} \de_k g_{p \ov q} \de_{\ov l} g_{i \ov j} + \chi^{k \ov l} g^{i \ov j} \de_k \de_{\ov l} g_{i \ov j},
\end{align*}
and
\begin{align*}
	\Laplace_\omega \tr \chi {\omega(t)} & = g^{i \ov j} \de_i \de_{\ov j} ( \chi^{k \ov l} g_{k \ov l} ) \\
	& = g^{i \ov j} \de_i \de_{\ov j} \chi^{k \ov l} g_{k \ov l} + g^{i \ov j} \chi^{k \ov l} \de_i \de_{\ov j} g_{k \ov l} + 2 \mathrm{Re}(g^{ i \ov j} \de_i \chi^{k \ov l } \de_{\ov j} g_{i \ov j} ),
\end{align*}
the difference gives
\begin{align*}
	\big( \frac \de {\de t} - \Laplace_\omega \big) \tr \chi \omega = & - g^{i \ov j} \de_i \de_{\ov j} \chi^{k \ov l} g_{k \ov l} + \chi^{k \ov l} g^{i \ov j} (\de_k \de_{\ov l} g_{i \ov j} - \de_i \de_{\ov j} g_{i \ov j} ) \\
	& - 2 \mathrm{Re}(g^{i \ov j} \de_i g^{k \ov l} \de_{\ov j} g_{k \ov l}) - \chi^{k \ov l} g^{p \ov j} g^{i \ov q} \de_k g_{p \ov q} \de_{\ov l } g_{i \ov j}.
\end{align*}
Then
\[
	\de_k \de_{\ov l} g_{i \ov j} - \de_i \de_{\ov j} g_{k \ov l} =  \de_k \de_{\ov l} ( \hat g_{i \ov j} - 2 t \Ric_{i \ov j} + \de_i \de_{\ov j} \vp) - \de_i \de_{\ov j} ( \hat g_{k \ov l} - 2 t \Ric_{k \ov l} + \de_k \de_{\ov l} \vp),
\]
but since $\Ric_{i \ov j}$ is the Chern-Ricci tensor of $\chi$, it satisfies
\[
	\de_k \de_{\ov l} \Ric_{i \ov j} = \de_i \de_{\ov j} \Ric_{k \ov l},
\]
and therefore
\begin{equation}\label{eq:15}
	\de_k \de_{\ov l} g_{i \ov j} - \de_i \de_{\ov j} g_{k \ov l} = \de_k \de_{\ov l} \hat g_{i \ov j} - \de_i \de_{\ov j} \hat g_{k \ov l}.
\end{equation}
We obtain the equality claimed above.

We now estimate the four terms above in succession.

\begin{lem}\label{eq:29}
There is a uniform constant $C_M>0$, depending only on $M$, such that:
\begin{center}
\begin{minipage}{.9\textwidth}
	\begin{enumerate}[(i)]
		\item $ - g^{i \ov j} \de_i \de_{\ov j} \chi^{k \ov l} g_{k \ov l} \leq - C_M^{-1} (\tr g \Ric) \tr \chi g - g^{i \ov j} g_{k \ov s} \chi_{r \ov l} \de_i \chi^{k \ov l} \de_{\ov j} \chi^{r \ov s} $
		\item $\chi^{k \ov l} g^{i \ov j}(\de_k \de_{\ov l} \hat g_{i \ov j} - \de_i \de_{\ov j} \hat g_{k \ov l}) \leq C_M \tr g {\Ric} + C_M \tr g \Theta$
		\item $- 2 \mathrm{Re}(g^{i \ov j} \de_i \chi^{k \ov l} \de_{\ov j} g_{k \ov l}) \leq C_M \tr g \Theta + \chi^{k \ov l} g^{p \ov j} g^{i \ov q} \de_k g_{p \ov q} \de_{\ov l} g_{i \ov j} + g^{i \ov j} g_{k \ov s} \chi_{r \ov l} \de_i \chi^{k \ov l} \de_{\ov j} \chi^{r \ov s}$
	\end{enumerate}
\end{minipage}
\end{center}
\end{lem}

From the stated estimates we obtain the following Corollary.
\begin{cor}\label{cor:31}
There is a uniform constant $C_M>0$ depending only on the geometry of $M$ such that
\[
	(\frac \de {\de t} - \Laplace) \tr \chi \omega \leq - \tr g \Ric \big( C_M^{-1} \tr \chi g - C_M) + C_M \tr g \Theta.
\]
\end{cor}

\begin{proof}[Proof of Lemma \ref{eq:29}]

To prove part $(i)$, we use that $\chi^{k \ov l} = \Phi^{2 k_l} \delta^{k l}$, so that
\begin{equation}\label{eq:17}
	\de_i \chi^{k \ov l} = \de_i (\Phi^{2k_l} \delta^{k \ov l}) = 2 k_l \Phi^{2 k_l} \delta^{k \ov l} \frac {\Phi_i} \Phi = 2k_l \chi^{k \ov l} \frac {\Phi_i} \Phi,
\end{equation}
and
\begin{align*}
	\de_i \de_j \chi^{ k \ov l} = \de_i \de_{\ov j}( \Phi^{2 k_l} \delta^{k l}) & = 2 k_l \Phi^{2 k_l} \delta^{k l} \big( \frac {\Phi_{i \ov j}} \Phi + (2 k_l -1) \frac {\Phi_i \Phi_{\ov j}} {\Phi^2} \big) \\
	& = 2 k_l \chi^{k \ov l} ( \hat g_{i \ov j} + (2 k_l - 1)\Theta_{i \ov j} )
\end{align*}
and by ~\eqref{eq:31},
\[
	\de_i \de_{\ov j} \chi^{k \ov l} = k_l \chi^{k \ov l} \Ric_{i \ov j} + 4 k_l^2 \chi^{k \ov l} \Theta_{i \ov j}.
\]
Finally, we note that by ~\eqref{eq:17}
\[
	\chi_{r \ov l} \de_i \chi^{k \ov l} \de_{\ov j} \chi^{r \ov s} = 4 k_l^2 \chi^{k \ov l} \Theta_{i \ov j},
\]
and then we obtain the inequality in part $(i)$ provided $C_M > \frac 2 {k_1}$.

Next, for the claim in part $(ii)$, we can take $C_M>0$ to be a constant large enough that
\[
	\chi^{k \ov l} \de_k \de_{\ov l} \hat g_{i \ov j} - \chi^{k \ov l} \de_i \de_{\ov j} \hat g_{k \ov l} \leq C_M \hat g_{i \ov j}.
\]
Since $\chi$ and $\hat g$ are fixed, it is clear that the constant depends only on the geometry of $M$. Now, using ~\eqref{eq:31},
\[
	g^{i \ov j}( \chi^{k \ov l} \de_k \de_{\ov l} \hat g_{i \ov j} - \chi^{k \ov l} \de_i \de_{\ov j} \hat g_{k \ov l}) \leq C_M \tr g \Ric + C_M \tr g \Theta
\]
which proves part $(ii)$.

Finally, moving on to part $(iii)$, we write
\[
	- 2 \mathrm{Re}(g^{i \ov j} \de_i \chi^{k \ov l} \de_{\ov j} g_{k \ov l}) = -2 \mathrm{Re}(g^{i \ov j} \de_i \chi^{k \ov l} \de_{\ov l} g_{k \ov j}) - 2 \mathrm{Re}(g^{i \ov j} \de_i \chi^{k \ov l}(\de_{\ov j} g_{k \ov l} - \de_{\ov l} g_{k \ov j})).
\]
For the second term, we have
\[
	\de_{\ov j} g_{k \ov l} - \de_{\ov l} g_{k \ov j} = \de_{\ov j} \hat g_{ k\ov l} - \de_{\ov l} \hat g_{k \ov j},
\]
(see the argument preceding \eqref{eq:15}). Then we use
\[
	\hat g_{i \ov j} = \frac {\Phi_{i \ov j}} \Phi,
\]
to obtain
\[
	\de_{\ov j} \hat g_{k \ov l} = \de_{\ov j}( \frac{ \Phi_{k \ov l}} \Phi ) = \frac {\Phi_{k \ov {l j}}} \Phi - \frac {\Phi_{k \ov l} \Phi_{\ov j}}{\Phi^2} ,
\]
so that
\[
	\de_{\ov j} \hat g_{k \ov l} - \de_{\ov l} \hat g_{k \ov j} = \frac {\Phi_{k \ov j} \Phi_{\ov l} - \Phi_{k \ov l} \Phi_{\ov j}} {\Phi^2}
\]
and using ~\eqref{eq:17}
\[
	\de_i \chi^{k \ov l} (\de_{\ov j} \hat g_{k \ov l} - \de_{\ov l} \hat g_{k \ov j}) = 2 k_l \chi^{k \ov l} (\hat g_{k \ov j} \Theta_{i \ov l} - \hat g_{k \ov l} \Theta_{i \ov j}),
\]
and therefore
\begin{align}
	g^{i \ov j} \de_i \chi^{k \ov l} (\de_{\ov j} \hat g_{k \ov l} - \de_{\ov l} \hat g_{k \ov j})
	 \leq & 2 g^{i \ov j} \chi^{k \ov l} \hat g_{k \ov j} \Theta_{i \ov l} \nonumber \\
	\leq & C_M g^{i \ov j} \Theta_{i \ov j} \label{eq:18}
\end{align}
after taking $C_M>0$ large enough that $4 \hat g \leq C_M \chi$. Again, the constant here depends only on $M$.

Now,
\begin{align*}
	-2 \mathrm{Re}(g^{i \ov j} \de_i \chi^{k \ov l} \de_{\ov l} g_{k \ov j})
	& = -2 \mathrm{Re}(g^{i \ov j} \chi^{p \ov l} g^{k \ov v} (\chi_{p \ov q} g_{u \ov v} \de_i \chi^{u \ov q}) \de_{\ov l} g_{k \ov j}) \\
	& \leq g^{i \ov j} \chi^{p \ov l} g^{k \ov v} (g_{u \ov v} \chi_{p \ov q} \de_i \chi^{u \ov q})(g_{k \ov s} \chi_{r \ov l} \de_{\ov j} \chi^{r \ov s}) + g^{i \ov j} \chi^{p \ov l} g^{k \ov v} \de_p g_{i \ov v} \de_{\ov l} g_{k \ov j} \\
	&  = g^{i \ov j} g_{k \ov s} \chi_{r \ov l} \de_i \chi^{k \ov l} \de_{\ov j} \chi^{r \ov s} + g^{i \ov j} \chi^{p \ov l} g^{k \ov v} \de_p g_{i \ov v} \de_{\ov l} g_{k \ov j},
\end{align*}
and then, combining with ~\eqref{eq:18}, we have
\[
	-2 \mathrm{Re}(g^{i \ov j} \de_i \chi^{k \ov l} \de_{\ov l} g_{k \ov j}) \leq C_M \tr g \Theta + g^{i \ov j} g_{k \ov s} \chi_{r \ov l} \de_i \chi^{k \ov l} \de_{\ov j} \chi^{r \ov s} + g^{i \ov j} \chi^{p \ov l} g^{k \ov v} \de_p g_{i \ov v} \de_{\ov l} g_{k \ov j}
\]
which proves the claim in part $(iii)$.
\end{proof}

\begin{remark}
	The presence of the $C_M \tr g \Theta$ term in Corollary \ref{cor:31} introduces difficulties  in applying the maximum principle argument. These difficulties are dealt with in the final step of proving Theorem ~\ref{thm:main}.
\end{remark}

Finally, we prove the main Theorem.
\begin{proof}[Proof of Theorem ~\ref{thm:main}]
Applying the previous Corollary, we arrive at
\begin{align*}
	(\frac \de {\de t} - \Laplace) \tr \chi \omega \leq - \tr g \Ric \big( C_M^{-1} \tr \chi g - C_M) + C_M \tr g \Theta
\end{align*}
Now, for large constants $A,B>0$ to be fixed later, define
\[
	Q = \tr \chi \omega  - A \vp - A (1 - 2 t) (\log (1 - 2t) - 1) - B t.
\]
Then we have the evolution inequality,
	\begin{align*}
		(\frac \de {\de t} - \Laplace) Q \leq  & - \tr \omega \Ric (C_M^{-1} \tr \chi \omega - C_M)  + C_M \tr \omega \Theta \\
		&  - A \dot \vp + 2 A \log (1 - 2 t) + A \tr \omega {(\omega - \omega_t)} - B \\
		= & - \tr \omega \Ric (C_M^{-1} \tr \chi \omega - C_M)  + C_M \tr \omega \Theta - (A-1) \tr \omega {\omega_t} \\
		& - A \log \frac {\omega^2}{\chi^2 (1 - 2t)^2} + 2A - \tr \omega {\omega_t} - B
	\end{align*}
	Next, by the arithmetic-geometric mean inequality,
	\[
		\tr \omega {\omega_t} \geq (1 - 2t) \tr \omega {\hat \omega} \geq A^{-1} \Big( \frac {(1- 2t)^2 \chi^2}{\omega^2} \Big)^{\frac 1 2}
	\]
	provided $A$ is taken sufficiently large. Furthermore, by Corollary ~\ref{eq:10cor}
	\[
		\omega_t = (1-2t) \hat \omega + 2 t \Theta \geq \Theta
	\]
	for all $t \geq 0$, and so we may fix $A = A(C_M)$ large enough that
	\[
	C_M \tr \omega \Theta - (A-1) \tr \omega {\omega_t} \leq 0.
	\] 
	Now, we have
	\begin{align*}
		(\frac \de {\de t} - \Laplace) Q \leq - & \tr \omega \Ric ( C_M^{-1} \tr \chi \omega - C_M) \\
		& + A \log \frac {(1-2t)^2 \chi^2}{\omega^2} - A^{-1} \Big( \frac {(1-2t)^2 \chi^2}{\omega^2} \Big)^{\frac 1 2} + 2A - B \\
		\leq - & \tr \omega \Ric (C_M^{-1} \tr \chi \omega - C_M),
	\end{align*}
	for $B=B(A)>0$ sufficiently large since $(A \log s - A^{-1} s^{\frac 1 2} + 2A)$ is bounded from above for $s>0$. It then follows that if $Q$ achieves a maximum with $t_0>0$, then at that point
	\[
		0 \leq - \tr \omega \Ric ( C_M^{-1} \tr \chi \omega - C_M),
	\]
	but then since $\Ric$ is non-negative, it follows that at the point of maximum
	\[
		\tr \chi \omega \leq C^2_M.
	\]
	Finally, since $\vp$, $Bt$, and $(1-2t) \log (1-2t)$ are all bounded, we obtain that $Q$ is bounded above on $M \times [0,\frac 12)$, and therefore
	\[
		\tr \chi \omega \leq C
	\]
	for a uniform constant $C>0$, which completes the proof.
\end{proof}

\end{document}